\UseRawInputEncoding 
\documentclass[12pt]{article}
\usepackage{graphicx}
\usepackage{amsmath}
\usepackage{amsthm}
\usepackage{lineno}
\usepackage{amsfonts,latexsym,rawfonts,amsmath,amssymb,amsthm,graphicx,color,amssymb}
\usepackage[colorlinks,citecolor=green,anchorcolor=blue,backref=page]{hyperref}
 \usepackage{indentfirst}
\numberwithin{equation}{section}
\theoremstyle{plain}
\newtheorem{thm}{Theorem}[section]
\newtheorem{lem}[thm]{Lemma}

\newtheorem{cor}[thm]{Corollary}

\theoremstyle{definition}
\newtheorem{defn}{Definition}

\newtheorem{rem}[thm]{Remark}

\allowdisplaybreaks

\def \d {\mathrm{d}}

\def \vol{\mathrm{Vol}}

\def\BiRic{\mathrm{BiRic}}
\def\Ric{\mathrm{Ric}}
\def\d{\mathrm{d}}
\def\Vol{\mathrm{Vol}}

\title{ \textbf{Vanishing theorem  and finiteness theorem for $p$-harmonic $1$ form} }
\author{XiangZhi Cao\thanks{School of information engineering, Nanjing Xiaozhuang University, Nanjing {\rm211171}, China.}\thanks{Email:aaa7756kijlp@163.com } }

\date{}
\begin{document}
\maketitle
\tableofcontents
\begin{abstract}
	In this paper, we will show vanishing theorem of $p$ harmonic $1$ form on submanifold $M$ in $ \bar{M} $ whose BiRic curvature satisfying $ \overline{\mathrm{BiRic}}^a \geq \Phi_a(H,S) $. As an corollary, we can get the corresponding theorem for $ p $ harmonic function and $ p $ harmonic map. We also investigate the finiteness problem of $p$ harmonic $1$ form on submanifold $M$ in $ \bar{M} $ whose BiRic curvature satisfying $ \overline{\mathrm{BiRic}}^a \geq -k^2 $.
\end{abstract}
\section{Introduction}

Let $ (M,g) $ be an Riemannian manifold. A  vector bundled valued differential $ k $ form $ \omega $ is called p harmonic $ k $ form if it satisfy 
\begin{equation*}
	\begin{split}
		\mathrm{d} \omega=0, \delta (|\omega|^{p-2}\omega)=0.
	\end{split}
\end{equation*}
 When p=2, it is reduced to harmonic $ k $ form. Liouville type property and finiteness of the vector space  for differential form are  two important topics in the research of differential forms.

 Han \cite{zbMATH06827101} studied obtain some vanishing and finiteness theorems for $ L^p $ $ p $-harmonic 1-forms on a locally conformally flat Riemmannian manifolds . In \cite{zbMATH07040685}, Han studied Liouville theorem for p harmonic 1 form on submaniold in sphere. In \cite{zbMATH06859577}, Han proved the finiteness theorem of the space $ p $ harmonic one form on submaniold in Hadmard manifold if the first eigenvalue satisfies is bounded below by some constants.
 In \cite{zbMATH06668693}, Dung  showed some vanishing type theorems for $ p $-harmonic $ l $-forms on such a manifoldwith a weighted Poincaré inequality.
Motivated by \cite{zbMATH06668693}, in \cite{zbMATH07361892}, Chao et al. investigated   $ p $-harmonic $ l $-forms on Riemannian manifolds with a weighted Poincaré inequality, and we get a vanishing type theorem. In \cite{afuni2015monotonicity}, Afuni studied the monotonicity of vector bundle valued harmonic form. In \cite{9}, Zhang used the Moser ineration to obtian the vanishing theorem for $ p $ harmonic $ 1 $ form on manifold with nonnegative Ricci curvature.

The energy functional of $p$ harmonic map is  defined by
\begin{equation*}
	E_{p}(u)=\int_{M}\frac{| \nabla u|^{p}}{2}d\nu_{g}.
\end{equation*}
whose Euler-lagrange equation is as follows:
\begin{equation}\label{bitensionfild}
	\tau_{p}(u)=div(|du|^{p-2}du),
\end{equation}

A map $u$ is called p-harmonic map if $\tau_{p}(u)=div(|du|^{p-2}du)=0$. 
 Motivated by \cite{wangql}, in \cite{MR3692378}, we studied the Liouville theorem of $ p $ harmonic map with finite energy  from complete noncompact submanifold in partially nonnegative curved manifold into nonpositive curved manifold, the conditons in our theorem is the index of the operator $ \Delta+\frac{1}{4}(S-nH^2) $ is zero or the small conditon on $ \|S-nH^2\|_{\frac{n}{2}} $.
In\cite{zbMATH06451355}, Han obtained liouville type theorem  for p harmonic function on submanifold in sphere.  Once can refer to \cite{Baird1992,Casteras2020,dung2017p,Li2001,pigola2008constancy,MR1145657,Zhang2016A,MR1966686} and reference therein for the researches on $ p $ harmonic map and $ p $ harmonic function.

In \cite{dung2020harmonic},  the authors defined the tensor
\begin{equation*}
	\begin{split}
		\overline{\mathrm{BiRic}}^a (u,v)=\mathrm{Ric}(u,u)+a\mathrm{Ric}(v,v)-K(u,v),
	\end{split}
\end{equation*}
and obtained vanishing theorem of harmonic  one form on submanifold M in $ \bar{M} $ whose BiRic curvature satisfies $ \overline{\mathrm{BiRic}}^a \geq \Phi_a(H,S) .$ They also proved that the space $ L^{2p} $ harmonic 1 forms on $ M $ is finite if $ \overline{\mathrm{BiRic}}^a \geq -k^2 $ , provided the first eigenvalue  is bounded below by a suitable constant and $ p $ shall satisfy some conditions.

Motivated by the conditions in \cite{dung2020harmonic}, in this paper, we will establish similar Theorem as \cite[Theorem 1.1, Theorem 4.4]{dung2020harmonic} for $ p $ harmonic 1 form on submanifold M in $ \bar{M} $ whose BiRic curvature satisfies $ \overline{\mathrm{BiRic}}^a \geq \Phi_a(H,S) $. As an corollary, we can get the corresponding theorem for $ p $ harmonic function and $ p $ harmonic map.

\section{Prelimilary}
Before stating our proof, we give some important formulas, definitions and some lemmas. We say a map $u$ is of $L^{q}$-finite energy if $\int_{M}|\nabla u|^{q}<\infty$.
  For $  p $-harmonic maps,  we have the Bochner formula(c.f. Lemma 2.4 in \cite{MR3692378})
\begin{equation}\label{bochner}
	\begin{split}
		  \frac{1}{2}\Delta|du|^{2p-2}
		& =|\nabla|du|^{p-2}du|^{2}-\langle |du|^{p-2}du,\triangle|du|^{p-2}du\rangle\\
		&+|du|^{2p-4}\langle du(\mathrm{Ric}^{M}(e_{k}),du(e_{k})\rangle\\
		&-|du|^{2p-4}\mathrm{R}^{N}(du(e_{i}),du(e_{k}),du(e_{i}),du(e_{k})).
	\end{split}
\end{equation}

\begin{defn}
  Space $H^{k,p}(L^{p}(M))=\{\omega\in A^{k}(M) :\d\omega=0;   \delta d(|\omega|^{p-2}\omega)=0 ;\int_{B_{R}}|\omega|^{p}=o(R^{\gamma}); 0<\gamma <2  ,\forall R>0.  \}$
\end{defn}

\begin{lem}[c.f. Lemma 2.7 in \cite{2} ]\label{1121}
    For  any section  $\omega\in \Gamma(A^{p}(M))$ which satisfies $\d \omega=0$ and any function $f$
on $M$,  we  have

\begin{equation}
\big|d(f\omega)\big|\leq |df||\omega|.
\end{equation}
	Let 
\begin{equation*}
	\begin{split}
		A_{p,n,q}=\begin{cases}
			 \frac{1}{\max\{q,n-q\}}	& \text{if}\quad  p=2\\
			\frac{1}{(p-1)^2}\min\{1, \frac{(p-1)^2}{n-1}\}	&  \text{if} \quad  p>2 \quad  \text{and} \quad q=1\\
			0	& \text{if}\quad   p>2 \quad  \text{and} \quad 1< q\leq n-1
		\end{cases}
	\end{split}
\end{equation*}
We can see that $ A_{2,n,1}=\frac{1}{n-1} $ if $ n \geq 2. $ Hereafter, we denote $ A_{p,n,1} $ by $ A_{p,n} .$

\begin{lem}[Kato's inequality, c.f. Lemma 2.2 in \cite{MR3849353} ]\label{kato}
	For $p \geq 2, q \geq 1$, let $\omega$ be an $p$-harmonic $q$-form on a complete Riemannian manifold $M^{n}$. The following inequality holds
	\[
	\left|\nabla\left(|\omega|^{p-2} \omega\right)\right|^{2} \geq\left.\left. (1+A_{p, n, q}) |\nabla| \omega\right|^{p-1}\right|^{2}.
	\]
	Moreover, when $p=2, q>1$ then the equality holds if and only if there exixts a 1-form $\alpha$ such that
	\[
	\nabla \omega=\alpha \otimes \omega-\frac{1}{\sqrt{q+1}} \theta_{1}(\alpha \wedge \omega)+\frac{1}{\sqrt{n+1-q}} \theta_{2}\left(i_{\alpha} \omega\right) .
	\]
	
\end{lem}

\end{lem}


\section{Liouville type theorem for $ p $ harmonic $ 1 $ form }
\begin{lem}[c.f. \cite{dung2020harmonic}]
	Let $M^{n}$ be an immersed hypersurface in a Riemannian manifold $\bar{M}^{n+1}$. Let Ric, $S, H$ denote the functions that assign to each point $p$ of $M$ the Ricci curvature, the square length of the second fundamental form, and the mean curvature respectively of $M$ at $p$, then for any tangent vector $X \in T_{p} M$, we have
	\[
	\operatorname{Ric}(X, X) \geq\left(\overline{\operatorname{BiRic}^{a}}\left(\frac{X}{|X|}, N\right)-\Phi_{a}(H, S)-a(\overline{\operatorname{Ric}}(N, N)+S)\right)|X|^{2},
	\]
	where
	\[
	\Phi_{a}(H, S)=\left(\frac{n-1}{n}-a\right) S-\frac{1}{n^{2}}\left\{2(n-1) H^{2}-(n-2) H \sqrt{(n-1)\left(n S-H^{2}\right)}\right\} .
	\]
\end{lem}
 We generalize Theorem 3.1 in \cite{dung2020harmonic}.

\begin{thm}\label{p1.1}
	Let $M^{m}(m \geq 2)$  be a complete, noncompact, connected, oriented, and
	stable hypersurface immersed in a Riemann manifold $ \bar{M} $. For any $ p\geq 2, $ if $ \overline{\mathrm{BiRic}}^a\geq \Phi_a(H,S) $, for some positive constant  a satisfying 
	\begin{equation*}
		\begin{split}
			a< 	\frac{4}{(m-1)p^2} +\frac{4(p-1)}{p^2}, 
		\end{split}
	\end{equation*}  
	then  there does not exist any nontrivial $ L^p  $ $ p $-harmonic 1-form on$  M $.

\end{thm}

\begin{proof}
 Without generality, we can assume $M-\{\omega(x)=0$, for $\forall x \in M\}\neq\varnothing$.  The proof below proceeds on $M^+=M-\{x \in M,\omega(x)=0\} $. Considering  the integral on $M$ is identical to that of $M$, we prefer to integrate function about $s$ on M in the subsequent computations.

According to Bochner formula for $ p $ harmopnic 1 form\cite{zbMATH06827101}:
\begin{align}\label{boch2}
  &\frac{1}{2}\Delta|\omega|^{2(p-1)} =|\nabla(|\omega|^{p-2}\omega)|^{2}-\big\langle\delta d (|\omega|^{p-2}\omega),|\omega|^{p-2}\omega\big\rangle+|\omega|^{2p-4} Ric(\omega,\omega).
\end{align}
On the other hand,
\begin{align}
\frac{1}{2}\Delta|\omega|^{2(p-1)}=|\omega|^{p-1} \Delta |\omega|^{p-1}+\left| \nabla\left|\omega\right|^{p-1}\right|^{2}.
\end{align}

Thus we have (c.f. \cite[(3)]{zbMATH06827101})
\begin{align*}
|\omega|^{p-1} \Delta |\omega|^{p-1}&=|\nabla(|\omega|^{p-2}\omega)|^{2}-| \nabla|\omega|^{p-1}|^{2}-(\delta d (|\omega|^{p-2}\omega),|\omega|^{p-2}\omega)+|\omega|^{2p-4} Ric(\omega,\omega)\\
&\geq \frac{1}{(m-1)(p-1)^2}| \nabla|\omega|^{p-1}|^{2}-\left(\delta d (|\omega|^{p-2}\omega),|\omega|^{p-2}\omega\right)+|\omega|^{2p-4} Ric(\omega,\omega),
\end{align*}
where we have used Kato inequality  Ln lemma 2.3 in \cite{2} in the first inequality.

Thus, we have
\begin{equation}\label{boch1}
	\begin{split}
		|\omega|\Delta |\omega|^{p-1}\geq&  \frac{4}{(m-1)p^2}|\nabla| \omega|^{\frac{p}{2}} |^{2}-\bigg\langle \delta d(|\omega|^{p-2}\omega),\omega\bigg\rangle\\
		&+|\omega|^{p}\left(\overline{\operatorname{BiRic}^{a}}\left(\frac{X}{|X|}, N\right)-\Phi_{a}(H, S)-a(\overline{\operatorname{Ric}}(N, N)+S)\right).
	\end{split}
\end{equation}

 It is well known that we  can choose cutoff function  $\phi  $ on  noncopact manifold $M$ such that
\begin{equation}
	\begin{cases}
		&  \text{  $0\leq \phi\leq 1$ } \\
		&  \text{$\phi=1$, on $B_{R}(x_{0})$    }\\
		&  \text{ $\phi=0$, on $ M- B_{2R}(x_{0})$, }\\
		&    \text{$ |\nabla\phi| < \frac{C}{R}$ },
\end{cases}   
\end{equation}

Multiplying both sides by $\phi^{2}$, and integating over $M$,
we have
\begin{equation*}
	\begin{split}
		\int_M \phi^2	|\omega|\Delta |\omega|^{p-1}\geq& \int_M \phi^2 \frac{4}{(m-1)p^2}|\nabla| \omega|^{\frac{p}{2}} |^{2}-\int_M \phi^2\bigg\langle \delta d(|\omega|^{p-2}\omega),\omega\bigg\rangle\\&
		+\int_M \phi^2|\omega|^{p}\left(\overline{\operatorname{BiRic}^{a}}\left(\frac{X}{|X|}, N\right)-\Phi_{a}(H, S)-a(\overline{\operatorname{Ric}}(N, N)+S)\right).
	\end{split}
\end{equation*}
which can be rearrange as
\begin{equation*}
	\begin{split}
		-&\int_M \phi^2 \nabla |\omega|\nabla|\omega|^{p-1}-2\int_M \phi\nabla|\omega| \nabla \phi	|\omega|^{p-1}\\
		\geq& \int_M \phi^2 \frac{4}{(m-1)p^2}|\nabla| \omega|^{\frac{p}{2}} |^{2}-\int_M \phi^2\bigg\langle \delta d(|\omega|^{p-2}\omega),\omega\bigg\rangle\\
		+&\int_M \phi^2|\omega|^{p}\left(\overline{\operatorname{BiRic}^{a}}\left(\frac{X}{|X|}, N\right)-\Phi_{a}(H, S)-aq\right).
	\end{split}
\end{equation*}
where $ q=(\overline{\operatorname{Ric}}(N, N)+S). $

Next, we need to deal with the last two terms on the right hand side of  \eqref{e88}. By the stable condition, we have 
\begin{equation*}
	\begin{split}
 \int_{M} q \varphi^{2}|\omega|^{p}\leq 	\int_{M}|\nabla \varphi|^{2}|\omega|^{p}+\frac{p^{2}}{4} \int_{M} \varphi^{2}|\omega|^{p-2}|\nabla| \omega|^{2}+p \int_{M} \varphi|\omega|^{p-1}\langle\nabla|\omega|, \nabla \varphi\rangle  .
	\end{split}
\end{equation*}
Thus, we have 
\begin{equation}\label{e88}
	\begin{split}
		&\frac{4}{(m-1)p^2}\int_M \phi^2 |\nabla| \omega|^{\frac{p}{2}} |^{2}\\
		\leq& -\int_M \phi^2 \nabla |\omega|\nabla|\omega|^{p-1}-2\int_M \phi|\omega|^{p-1}\nabla|\omega| \nabla \phi	\\
		+&\int_M \phi^2\bigg\langle \delta d(|\omega|^{p-2}\omega),\omega\bigg\rangle\\
				&+a\bigg[\int_{M}|\nabla \varphi|^{2}|\omega|^{p}+\frac{p^{2}}{4} \int_{M} \varphi^{2}|\omega|^{p-2}|\nabla| \omega|^{2}+p \int_{M} \varphi|\omega|^{p-1}\langle\nabla|\omega|, \nabla \varphi\rangle\bigg].
	\end{split}
\end{equation}
Notice that $ \phi^2 \nabla |\omega|\nabla|\omega|^{p-1}=\frac{4(p-1)}{p^2}|\nabla |\omega|^{\frac{p}{2}} |^2 $, thus we get 
\begin{equation}\label{yt}
	\begin{split}
		\frac{4}{(m-1)p^2}\int_M \phi^2 |\nabla| \omega|^{\frac{p}{2}} |^{2}\leq& -\int_M \phi^2 \frac{4(p-1)}{p^2}|\nabla |\omega|^{\frac{p}{2}} |^2+(ap-2)\int_M \phi|\omega|^{p-1}\nabla|\omega| \nabla \phi	\\
		+&\int_M \phi^2\bigg\langle \delta d(|\omega|^{p-2}\omega),\omega\bigg\rangle\\
		+&a\bigg[\int_{M}|\nabla \varphi|^{2}|\omega|^{p}+ \int_M \phi^2|\nabla |\omega|^{\frac{p}{2}} |^2\bigg].
	\end{split}
\end{equation}
Now we need to estimate some terms  on the right hand side of \eqref{yt}. By \cite[(9)]{2}, we know that 
\begin{equation*}
	\begin{split}
		&\int_{M}\bigg\langle d ( |\omega|^{p-2}\omega), d  \omega\bigg\rangle \leq 
	4\frac{p-2}{p}\int_{M} \Big| \nabla |\omega|^{\frac{p}{2}}\Big|  \phi |\nabla \phi||\omega|^{\frac{p}{2}},
	\end{split}
\end{equation*}
Using Young's inequality, we get 
\begin{equation}
	\begin{split}
		&\frac{4}{(m-1)p^2}\int_M \phi^2 |\nabla| \omega|^{\frac{p}{2}} |^{2}+\int_M \phi^2 \frac{4(p-1)}{p^2}|\nabla |\omega|^{\frac{p}{2}} |^2\\
		&+\int_M \phi^2|\omega|^{p}\left(\overline{\operatorname{BiRic}^{a}}\left(\frac{X}{|X|}, N\right)-\Phi_{a}(H, S)\right)\\
		\leq& \int_M \frac{(2-a p)^2}{4 \epsilon}|\nabla \varphi|^2|\omega|^p+ \int_M\epsilon \frac{4}{p^2} \phi^2|\nabla |\omega|^{\frac{p}{2}} |^2	\\
			+&\bigg[(a+\frac{(p-2)^2}{4\epsilon_3(p^2)}16)\int_{M}|\nabla \varphi|^{2}|\omega|^{p}+(a+\epsilon_3) \int_M \phi^2|\nabla |\omega|^{\frac{p}{2}} |^2\bigg].
	\end{split}
\end{equation}
which can be written as 
\begin{equation}\label{563}
	\begin{split}
	&\int_M \phi^2|\omega|^{p}\left(\overline{\operatorname{BiRic}^{a}}\left(\frac{X}{|X|}, N\right)-\Phi_{a}(H, S)\right)\\
	\leq &\left(\epsilon \frac{4}{p^2}+(a+\epsilon_3)-	\frac{4}{(m-1)p^2} -\frac{4(p-1)}{p^2}\right)  \int_M \phi^2|\nabla |\omega|^{\frac{p}{2}} |^2	\\
		+&\bigg[(a+\frac{(p-2)^2}{4\epsilon_3(p^2)}16)+\frac{(2-a p)^2}{4 \epsilon}\bigg]\int_{M}|\nabla \varphi|^{2}|\omega|^{p}.
	\end{split}
\end{equation}

Noticing that $ \overline{\operatorname{BiRic}^{a}}\left(\frac{X}{|X|}, N\right)\geq \Phi_{a}(H, S) $, let $R\rightarrow\infty$,  we have
\begin{equation*}
	\begin{split}
			0\leq \left(\epsilon \frac{4}{p^2}+(a+\epsilon_3)-	\frac{4}{(m-1)p^2} -\frac{4(p-1)}{p^2}\right)  \int_M \phi^2|\nabla |\omega|^{\frac{p}{2}} |^2.	
	\end{split}
\end{equation*}

Since $ a-	\frac{4}{(m-1)p^2} -\frac{4(p-1)}{p^2} <0, $ we can choose small $ \epsilon, \epsilon_{3} $ such that  $ \epsilon \frac{4}{p^2}+(a+\epsilon_3)-	\frac{4}{(m-1)p^2} -\frac{4(p-1)}{p^2}<0. $
 so, $ |\omega| $ is a constant, all the inequalities in  \eqref{boch2} and \eqref{boch1} becomes equality. Suppose that $ |\omega|\neq 0, $
by \eqref{boch2} and \eqref{boch1}, we have
  \begin{equation*}
  	\begin{split}
  		\Ric (\omega^*,\omega^*)=0.
  	\end{split}
  \end{equation*}
then, as in \cite{dung2020harmonic} , we see that  for any unit tangent vector $ X, $
\begin{equation*}
	\begin{split}
		\Ric(X,X)\geq 0.
	\end{split}
\end{equation*}
Then $ \vol(M)=\infty $, this gives a contradiction to  that $ |\omega|\in L^p. $

\end{proof}

Similar to \cite[Corollary 3.3]{dung2020harmonic}, we can establish
\begin{cor}
	 Let $M^n$ be a complete, noncompact, stable, and minimal hypersurface immersed in a Riemannian manifold with non-negative $\overline{B i R i c}^{a}$ curvature for some
	\[
	\frac{n-1}{n} \leq a<\frac{4}{(n-1)p^2} +\frac{4(p-1)}{p^2}  ,
	\]
	then  there does not exist any nontrivial $ L^p  $ $ p $-harmonic 1-form on$  M $. 
\end{cor}
\begin{proof}
By the argument before \cite[Corollary 3.3]{dung2020harmonic}, we know that if $ H=0 $, then $  \Phi_{a}(H,S)=(\frac{n-1}{n}-a)S $
\end{proof}

Similar to \cite[Corollary 3.4]{dung2020harmonic}, we can obtain
\begin{cor}
 Let $M^{n}$ be a complete, noncompact, stable immersed hypersurface in a Riemannian manifold $\bar{M}$. If one of the following conditions holds, then there is no nontrivial $L^{p}$ harmonic 1-form on $M$ :
 
	(1) $\overline{\operatorname{BiRic}^{1}} \geq \frac{n-5}{4} H^{2}$ and $p<2+\frac{2}{\sqrt{n-1}}$;
	
	(2) $\overline{\operatorname{BiRic}^{1}} \geq \Phi_{1}(H, S)$ and $p<2+\frac{2}{\sqrt{n-1}}$;
	
	(3) $\overline{{ \BiRic }}^{1} \geq 0,2 \leq n \leq 5$ and $p<2+\frac{2}{\sqrt{n-1}}$;
	
	(4) $\overline{\operatorname{BiRic}^{a}} \geq 0,  \frac{\sqrt{n-1}}{2} \leq a<\frac{4}{(n-1)p^2} +\frac{4(p-1)}{p^2}$,
	
	then  there does not exist any nontrivial $ L^p  $ $ p $-harmonic $ 1 $-form on $  M $. 
\end{cor}
\begin{proof}
	By \cite[(3.7)]{dung2020harmonic}, we know that \begin{equation*}
		\begin{split}
			\Phi_{1}(H,S)\leq \frac{n-5}{4}H^2.
		\end{split}
	\end{equation*}
Then in Theorem \ref{p1.1}, take $ a=1, $ we can finish the proof as in \cite[ Corollary 3.4]{dung2020harmonic}.  The statement follows from the fact that $  \Phi_{a}(H,S)\leq (\frac{\sqrt{n-1}}{2}-a)S $.
\end{proof}
Similar to \cite[Corollary 3.5]{dung2020harmonic}, we can obtain
\begin{cor}
	Let $M^{n},(n \geq 3)$ be a complete, compact, stable immersed hypersurface in a Riemannian manifold $\bar{M}$. Suppose that one of the following conditions holds
	
	(1) $\overline{\operatorname{BiRic}^{a}} \geq \frac{n-5}{4} H^{2}$ and $a=1, p<2+\frac{2}{\sqrt{n-1}}$;
	
	(2) $\overline{\operatorname{BiRic}^{a}} \geq 0,  \frac{\sqrt{n-1}}{2} \leq a<\frac{4\left(p-\frac{n-2}{n-1}\right)}{p^{2}}$.
	
	If $M$ admits a nontrivial $p$-harmonic 1-form $\omega$, then $\omega$ is parallel and $M$ has at least $n-1$ principal curvatures which are equal. Moreover, if $H=0$ then $M$ is totally geodesic.
\end{cor}
\begin{proof}
Follow the proof of 	\cite[Corollary 3.5]{dung2020harmonic}, we have 
\begin{equation*}
	\begin{split}
		0\leq \left(\epsilon \frac{4}{p^2}+(a+\epsilon_3)-	\frac{4}{(m-1)p^2} -\frac{4(p-1)}{p^2}\right)  \int_M |\nabla |\omega|^{\frac{p}{2}} |^2.	
	\end{split}
\end{equation*}
The rest proof is the sam as  that  of 	\cite[Corollary 3.5]{dung2020harmonic}.
\end{proof}
\begin{cor}\label{cor1}
	Let $M^{m}(m \geq 2)$  be an $m$-be a complete, noncompact, connected, oriented, and
	stable hypersurface immersed in a Riemann manifold $ \bar{M} $. For any $ p>0, $ if $ \overline{\mathrm{BiRic}}^a\geq \Phi_a(H,S) $, for some positive constant  a satisfying 
	\begin{equation*}
		\begin{split}
			a< \frac{4(p-1)^2}{(m-1)p^2} +\frac{4(p-1)}{p^2} .
		\end{split}
	\end{equation*}  
 Let $u$ be an $  p $-harmonic function on  $ M. $  For $p \geq 2$, if $u$ has finite  $ p $-energy, then $ u $ must be a constant map.
\end{cor}
\begin{rem}
	One can use the conditions and methods of \cite[Theorem 2.3]{dung2017p} to conclude that every $ p $ harmonic function u with finite $ L^{2\beta} $ energy is constant provided $ \beta $ satisfying some certain conditons. 
\end{rem}
\begin{proof}
	First, we recall the Bochner formula for $ p $ harmonic function (c.f. \cite{MR3692378} or \cite{zbMATH06451355} )
	\begin{align} \label{2.1.1}
		\frac{1}{2}\Delta|du|^{2p-2} =&|\nabla\left( |du|^{p-2}du\right) |^{2}-\langle|du|^{2p-2}du,\Delta\left( |du|^{p-2}du\rangle\right) \nonumber\\
		&+|du|^{2p-4}\langle du(\mathrm{Ric}^{M}(e_{k}),du(e_{k})\rangle,
	\end{align}

the following  Kato inequality for p-harmonic function (cf. lemma 2.4 in \cite{zbMATH06451355}):
\begin{equation}\label{2.1.2}
	|\nabla(|du|^{p-2}du) |^2\geq \frac{n}{n-1}|\nabla|du|^{p-1} |^2.
\end{equation}
In this case, \eqref{boch1} can be replaced by

\begin{equation}
	\begin{split}
		|\omega|\Delta |\omega|^{p-1}\geq&  \frac{4(p-1)^2}{(m-1)p^2}|\nabla| \omega|^{\frac{p}{2}} |^{2}-\bigg\langle \delta d(|\omega|^{p-2}\omega),\omega\bigg\rangle\\
		&+|\omega|^{p}\left(\overline{\operatorname{BiRic}^{a}}\left(\frac{X}{|X|}, N\right)-\Phi_{a}(H, S)-a(\overline{\operatorname{Ric}}(N, N)+S)\right).
	\end{split}
\end{equation}
The rest proof  is almost the same as in the proof of Theorem \ref{p1.1} below \eqref{boch1}, once we replace $ \frac{4}{(m-1)p^2} $ by $ \frac{4(p-1)^2}{(m-1)p^2} $ .

\end{proof}

\begin{cor}\label{cor2}
	Let $M^{n}(n \geq 2)$  be an $m$-be a complete, noncompact, connected, oriented, and
	stable hypersurface immersed in a Riemann manifold $ \bar{M} $. For any $ p>0, $ if $ \overline{\mathrm{BiRic}}^a\geq \Phi_a(H,S) $, for some positive constant  a satisfying 
	\begin{equation*}
		\begin{split}
			a< \frac{4(p-1)}{p^2} 
		\end{split}
	\end{equation*}  
	Let $u:\left(M, g,  d v_{g}\right) \rightarrow(N, h)$ be an p-harmonic map from an oriented complete noncompact manifold into a Riemannian manifold  and $K^{N} \leq 0$. For $p \geq 2$, if $u$ has finite  $ p $-energy, then $ u $ must be a constant map.
\end{cor}
\begin{proof}
	It is well known that (c.f. \cite{MR3692378})
\begin{align} \label{2.1.1}
	\frac{1}{2}\Delta|du|^{2p-2} =&|\nabla\left( |du|^{p-2}du\right) |^{2}-\langle|du|^{2p-2}du,\Delta\left( |du|^{p-2}du\rangle\right) \nonumber\\
	&+|du|^{2p-4}\langle du(\mathrm{Ric}^{M}(e_{k}),du(e_{k})\rangle+|du|^{2p-4}\langle R^N( du(e_i),du(e_j),du(e_i),du(e_j))\rangle,
\end{align}

However, for $ p $ harmonic map, we only have (\cite{zbMATH04182289})
\begin{equation*}
	\begin{split}
		|\nabla(|du|^{p-2}du) |^2\geq |\nabla|du|^{p-1} |^2.
	\end{split}
\end{equation*}

In this case,  in the proof of Theorem \ref{p1.1},\eqref{boch1} can be replaced by
\begin{equation}
	\begin{split}
		|\omega|\Delta |\omega|^{p-1}\geq&  -\bigg\langle \delta d(|\omega|^{p-2}\omega),\omega\bigg\rangle\\
		&+|\omega|^{p}\left(\overline{\operatorname{BiRic}^{a}}\left(\frac{X}{|X|}, N\right)-\Phi_{a}(H, S)-a(\overline{\operatorname{Ric}}(N, N)+S)\right).
	\end{split}
\end{equation}
The rest proof  is almost the same as in the proof of Theorem \ref{p1.1} below \eqref{boch1}, once we replace $ \frac{4}{(m-1)p^2} $ by zero .
\end{proof}
\begin{defn}[c.f. Definition 3.6 in \cite{dung2020harmonic}]
	 An immersed hypersurface $M^{n}$ in a Riemannian manifold $\bar{M}^{n+1}$ is said to have a Sobolev inequality if there exists a positive constant $C_{s}$ such that
\begin{equation}\label{sobo}
	\begin{split}
		\left(\int_{M} f^{\frac{2 n}{n-2}}\right)^{\frac{n-2}{n}} \leq C_{s} \int_{M}|\nabla f|^{2}
	\end{split}
\end{equation}
	for any nonnegative $\mathcal{C}^{1}$-functions $f: M \rightarrow \mathbb{R}$ with compact support. Here, $C_{s}$ is said to be the Sobolev constant.
\end{defn}
\begin{thm}
	Let $M^{n}(n \geq 2)$  be an $n$-dimensional a complete, noncompact, connected, oriented, and
	stable minimal  hypersurface immersed in a Riemann manifold $ \bar{M} $.Assume that a
	Sobolev inequality holds on M . For any $ p\geq 2, $ if $ \bar{BiRic^a}\geq 0 $, for some positive constant  $ a $ satisfying 
	\begin{equation*}
		\begin{split}
			a< \min \bigg\{\frac{4}{(n-1)p^2} +\frac{4(p-1)}{p^2} , \frac{n-1}{n} \bigg\} ,
		\end{split}
	\end{equation*} 
and 
\begin{equation}\label{}
	\begin{split}
		\|S\|_{\frac{n}{2}} \leq  \frac{\frac{4}{(n-1)p^2} +\frac{4(p-1)}{p^2} -a}{C_s(\frac{n-1}{n}-a)},
	\end{split}
\end{equation} 
where $ C_s $ is the constant in \eqref{sobo},
	then  there does not exist any nontrivial $ L^p  $ p-harmonic 1-form on M. 
\end{thm}
\begin{proof}From \eqref{563} in the proof of Theorem \ref{p1.1}, it is easy to see that 
 \begin{equation}\label{er}
 	\begin{split}
 		&\left(\epsilon \frac{4}{p^2}+(a+\epsilon_3)-	\frac{4}{(n-1)p^2} -\frac{4(p-1)}{p^2}\right)  \int_M \phi^2|\nabla |\omega|^{\frac{p}{2}} |^2	\\
 		&+\bigg[(a+\frac{(p-2)^2}{4\epsilon_3(p^2)}16)+\frac{(2-a p)^2}{4 \epsilon}\bigg]\int_{M}|\nabla \varphi|^{2}|\omega|^{p}\\
 	\geq&	\int_M \phi^2|\omega|^{p}\left(\overline{\operatorname{BiRic}^{a}}\left(\frac{X}{|X|}, N\right)-\Phi_{a}(H, S)\right)\\
 	\geq &-C_{n,a}  \int_M S\phi^2|\omega|^{p}.
 	\end{split}
 \end{equation}
By \cite{dung2020harmonic}, we have 
\begin{equation*}
	\begin{split}
	\int_M S\phi^2|\omega|^{p}	\leq C_s\|S\|_{n / 2}\left(\left(1+\frac{1}{\epsilon}\right) \int_M|\nabla \varphi|^2|\omega|^p+\frac{(1+\epsilon) p^2}{4} \int_M \varphi^2|\omega|^{p-2}|\nabla| \omega \|^2\right).
	\end{split}
\end{equation*}
Thus ,we have 

\begin{equation}\label{}
	\begin{split}
	0\leq& 	\left(\epsilon \frac{4}{p^2}+(a+\epsilon_3)-	\frac{4}{(n-1)p^2} -\frac{4(p-1)}{p^2}+C_{n,a}C_s\|S\|_{\frac{n}{2}}(1+\epsilon)\right)  \int_M \phi^2|\nabla |\omega|^{\frac{p}{2}} |^2	\\
		&+\bigg[(a+\frac{(p-2)^2}{4\epsilon_3(p^2)}16)+\frac{(2-a p)^2}{4 \epsilon}+C_{n,a}C_s\|S\|_{n / 2}\left(1+\frac{1}{\epsilon}\right)\bigg]\int_{M}|\nabla \varphi|^{2}|\omega|^{p}.
	\end{split}
\end{equation}
Notice that Sobolev inequality holds on $ M $ , then $ \mathrm{vol}(M)=\infty. $ Then by the same argument as in the proof of Theorem \ref{p1.1}, we see that $ \omega $ is trivial.
\begin{cor}
	Let $M^{n}(n \geq 2)$  be an $n$-dimensional a complete, noncompact, connected, oriented, and
	stable minimal  hypersurface immersed in a Riemann manifold $ \bar{M} $.Assume that a
	Sobolev inequality holds on M . For any $ p>\frac{n-2}{n-1}, $ if $ \overline{\BiRic}^a\geq 0 $, for some positive constant  a satisfying 
	\begin{equation*}
		\begin{split}
			a< \min \bigg\{\frac{4}{(n-1)p^2} +\frac{4(p-1)}{p^2} , \frac{n-1}{n} \bigg\} 
		\end{split}
	\end{equation*} 
	and 
	\begin{equation}\label{}
		\begin{split}
			\|S\|_{\frac{n}{2}} \leq  \frac{\frac{4}{(m-1)p^2} +\frac{4(p-1)}{p^2} -a}{C_s(\frac{n-1}{n}-a)}
		\end{split}
	\end{equation} 
 Let $u$ be an $  p $-harmonic function on  $ M. $  For $p \geq 2$, if $u$ has finite  $ p $-energy, then $ u $ must be a constant map.
\end{cor}

\begin{cor}
	Let $M^{n}(n \geq 2)$  be an $n$-dimensional a complete, noncompact, connected, oriented, and
	stable minimal  hypersurface immersed in a Riemann manifold $ \bar{M} $. Assume that a
	Sobolev inequality holds on M . For any $ p\geq 2, $ if $ \overline{\BiRic}^a\geq 0 $, for some positive constant  a satisfying 
	\begin{equation*}
		\begin{split}
			a< \min \bigg\{\frac{4}{(n-1)p^2} +\frac{4(p-1)}{p^2} , \frac{n-1}{n} \bigg\} 
		\end{split}
	\end{equation*} 
	and 
	\begin{equation}\label{}
		\begin{split}
			\|S\|_{\frac{n}{2}} \leq  \frac{\frac{4}{(n-1)p^2} +\frac{4(p-1)}{p^2} -a}{C_s(\frac{n-1}{n}-a)}
		\end{split}
	\end{equation} 
	Let $u:\left(M, g, d v_{g}\right) \rightarrow(N, h)$ be an p-harmonic map from an oriented complete noncompact manifold into a Riemannian manifold  and $K^{N} \leq 0$. For $p \geq 2$, if $u$ has finite  $ p $-energy, then $ u $ must be a constant map.
\end{cor}
\begin{proof}
	As in the proof of Corollary \ref{cor2},  we replace $ \frac{4}{(n-1)p^2} $ by zero in \eqref{boch1} and below.
\end{proof}

\end{proof}
\begin{thm}\label{wt}
 Let $M^n(n \geq 2)$ be a complete, noncompact, oriented, and stable immersed hypersurface in a Riemann manifold $\bar{M}$. Assume that $M$ satisfies a Sobolev inequality,
	$$
	\left(\int_M|f|^{\frac{2 n}{n-2}}\right)^{\frac{n-2}{n}} \leq C_s \int_M|\nabla f|^2,
	$$
	for any smooth compactly supported function $f$ in $M$. For any $p\geq 2$ , if $\overline{\operatorname{BiRic}}^a \geq 0$ for some positive constant a satisfying
\begin{equation*}
	\begin{split}
		a< \min \bigg\{\frac{4}{(n-1)p^2} +\frac{4(p-1)}{p^2} , \frac{\sqrt{n-1}}{2} \bigg\} 
	\end{split}
\end{equation*} 
and 
\begin{equation}\label{}
	\begin{split}
		\|S\|_{\frac{n}{2}} \leq  \frac{\frac{4}{(n-1)p^2} +\frac{4(p-1)}{p^2} -a}{C_s(\frac{\sqrt{n-1}}{2}-a)}
	\end{split}
\end{equation} 
	then there does not exist any nontrivial $L^p$ $ p $-harmonic 1-form on $M$.
\end{thm}

\begin{proof}
	In \eqref{er}, notice that 
	\begin{equation*}
		\begin{split}
			\Phi(H,s)\geq C_{n,a}=\frac{\sqrt{n-1}}{2}-a
		\end{split}
	\end{equation*}
The proof is the same as in Theorem \ref{wt}.
\end{proof}

\section{ Finiteness Theorem}

In this paper, we will use  the method of  \cite{2} and \cite{dung2020harmonic} to prove  that the space of $L^{p}$ harmonic 1-forms on $M$ is finite. Firstly, we consider the case where $ H\neq 0. $
\begin{thm}\label{p1.2}
	Let $\bar{M}$ be an $(n+1)$-dimensional Hadamard manifold with $\overline{\BiRic}^{a}$ satisfying $-k^{2} \leq \overline{\operatorname{BiRic}}^{a}$, where $  a $ is a given nonnegative real number, $k$ is a nonzero constant, and $3 \leq n \leq 4$. Let $M$ be a complete noncompact hypersurface with finite index that is immersed in $\bar{M}$. If one of the following two conditions hold,
	
 (1)	If  $\frac{2}{2-\sqrt{n-1}}<p<\min\{1+\sqrt{n-1},2\frac{n-2}{n-3}\}$, assume that
	\[
	\lambda_{1}(M)>\left(	\frac{2-p}{p}+ \frac{1}{n-1} \frac{2(p-1)}{p}-\epsilon \right) ^{-1}\left( \frac{p}{2p-2}+\frac{1}{a+1}\right)k^2.
	\]
where $ \epsilon $ is small positive constant.	

(2) If $   \max\{\frac{2}{2-\sqrt{n-1}},1+\sqrt{n-1}\} \leq p <3, $  assume that
	\[
	\lambda_{1}(M)>\left(\frac{2-p}{p}+\frac{2}{p(p-1)}-\epsilon \right)^{-1}\left( \frac{p}{2p-2}+\frac{1}{a+1}\right)k^2.
	\]
where $ \epsilon $ is small positive constant.	
	Then,
	\[
	\operatorname{dim} \mathcal{H}^{1}\left(L^{ p}(M)\right)<\infty,
	\]
	where $\mathcal{H}^{1}\left(L^{ p}(M)\right)$ denotes the space of $L^{p}$ harmonic 1-forms on $M$.
\end{thm}
\begin{rem}
	The conclusion is different from that in  \cite{dung2020harmonic} since we consider the space of $L^{p}$ $ p $-harmonic 1-forms on $M$ instead of  the space of $L^{2p}$ harmonic 1-forms.
	Although in \cite{2}, Han consider also the finiteness problem of the space of $L^{p}$ $ p $-harmonic 1-forms, the assumptions in our theorem is different from that in \cite{2}.
\end{rem}
\begin{proof}
	
By Lemma \eqref{kato}, it is easy to see that

\begin{equation}\label{}
	\begin{split}
				|\omega|^{p-1}\Delta |\omega|^{p-1}\geq& A_{p,n}|\nabla| \omega|^{p-1} |^{2}-\bigg\langle \delta d(|\omega|^{p-2}\omega),|\omega|^{p-2}\omega\bigg\rangle\\
		+&|\omega|^{2p-4}\left(\overline{\operatorname{BiRic}^{a}}\left(\frac{X}{|X|}, N\right)-\Phi_{a}(H, S)-a(\overline{\operatorname{Ric}}(N, N)+S)\right)|\omega|^{2}.
	\end{split}
\end{equation}

Computing directly, we have
\begin{equation}\label{eq3}
	\begin{split}
	&	|\omega|^{\frac{p}{2}} \Delta|\omega|^{\frac{p}{2}}\\
		=&\frac{2-p}{p}|\nabla| \omega|^{\frac{p}{2}}|^{2}+\frac{p}{2p-2}|\omega|^{2-p}\left( |\omega|^{p-1} \Delta|\omega|^{p-1}\right) \\
		=&\frac{2-p}{p}|\nabla| \omega|^{\frac{p}{2}}|^{2}+\frac{p}{2p-2}|\omega|^{2-p}\bigg[ A_{p,n}|\nabla| \omega|^{p-1} |^{2}-\bigg\langle \delta d(|\omega|^{p-2}\omega),|\omega|^{p-2}\omega\bigg\rangle\\
		&+|\omega|^{2p-4}\left(\overline{\operatorname{BiRic}^{a}}\left(\frac{X}{|X|}, N\right)-\Phi_{a}(H, S)-a(\overline{\operatorname{Ric}}(N, N)+S)\right)|\omega|^{2}\bigg] \\
		=&	\frac{2-p}{p}|\nabla| \omega|^{\frac{p}{2}}|^{2}+\frac{p}{2p-2} A_{p,n} \frac{4(p-1)^2}{p^2}|\nabla| \omega|^{\frac{p}{2}} |^{2}-\frac{p}{2p-2}\bigg\langle \delta d(|\omega|^{p-2}\omega),\omega\bigg\rangle\\
		&+\frac{p}{2p-2}\left(\overline{\operatorname{BiRic}^{a}}\left(\frac{X}{|X|}, N\right)-\Phi_{a}(H, S)-a(\overline{\operatorname{Ric}}(N, N)+S)\right)|\omega|^{p}\\
		\geq& 	\left( \frac{2-p}{p}+\frac{p}{2p-2} A_{p,n} \frac{4(p-1)^2}{p^2}\right) |\nabla| \omega|^{\frac{p}{2}} |^{2}-\frac{p}{2p-2}\bigg\langle \delta d(|\omega|^{p-2}\omega),\omega\bigg\rangle\\
		&+\frac{p}{2p-2}\left(-k^2-\frac{\sqrt{n-1}}{2}S\right)|\omega|^{p}.
	\end{split}
\end{equation}
%
%

By the stable condition, we have 
\begin{equation}\label{eee}
	\begin{split}
		\int_{M \backslash B_{R}(o)}(\overline{\operatorname{Ric}}(N)+S) \varphi^{2}|\omega|^{ p}  \leq& \int_{M \backslash B_{R}(o)}\left|\nabla\left(\varphi|\omega|^{\frac{p}{2}}\right)\right|^{2} \\
		=&\int_{M \backslash B_{R}(o)} \varphi^{2}|\nabla| \omega|^{\frac{p}{2}}|^{2}+\int_{M \backslash B_{R}(o)}|\nabla \varphi|^{2}|\omega|^{p} \\
		&+2 \int_{M \backslash B_{R}(o)} \varphi|\omega|^{\frac{p}{2}}\left\langle\nabla \varphi, \nabla|\omega|^{\frac{p}{2}}\right\rangle.
	\end{split}
\end{equation}

By \eqref{eee} and divergence theorem,  we have 
\begin{equation}\label{}
	\begin{split}
		&\int_{M \backslash B_{R}(o)}\left(\frac{-k^{2}}{a+1}+S\right) \varphi^{2}|\omega|^{ p}\\
		 \leq &\int_{M \backslash B_{R}(o)}|\nabla \varphi|^{2}|\omega|^{ p}-\int_{M \backslash B_{R}(o)} \varphi^{2}|\omega|^{\frac{p}{2}} \Delta|\omega|^{\frac{p}{2}}\\
		 \leq &\int_{M \backslash B_{R}(o)}|\nabla \varphi|^{2}|\omega|^{p}
		 -\int_{M \backslash B_{R}(o)} \varphi^{2}	\left( \frac{2-p}{p}+\frac{p}{2p-2} A_{p,n} \frac{4(p-1)^2}{p^2}\right) |\nabla| \omega|^{\frac{p}{2}} |^{2}\\
		 &+\int_{M \backslash B_{R}(o)}\frac{p}{2p-2}\varphi^{2}\bigg\langle \delta d(|\omega|^{p-2}\omega),\omega\bigg\rangle
		 +\int_{M \backslash B_{R}(o)}\frac{p}{2p-2}\varphi^{2}\left(k^2+\frac{\sqrt{n-1}}{2}S\right)|\omega|^{p},
	\end{split}
\end{equation}
where in the second inequality, we have used \eqref{eq3}.

Thus, we get
\begin{equation}\label{e64}
	\begin{split}
		&\int_{M \backslash B_{R}(o)} \varphi^{2}	\left( \frac{2-p}{p}+\frac{p}{2p-2} A_{p,n} \frac{4(p-1)^2}{p^2}\right) |\nabla| \omega|^{\frac{p}{2}} |^{2}\\
		\leq& \int_{M \backslash B_{R}(o)}|\nabla \varphi|^{2}|\omega|^{p}+\int_{M \backslash B_{R}(o)}\frac{p}{2p-2}\varphi^{2}\bigg\langle \delta d(|\omega|^{p-2}\omega),\omega\bigg\rangle
	\\
	&+\int_{M \backslash B_{R}(o)}\varphi^{2}\left(\frac{p}{2p-2}\frac{\sqrt{n-1}}{2}-1\right)S|\omega|^{p}+\int_{M \backslash B_{R}(o)}\left( \frac{p}{2p-2}+\frac{1}{a+1}\right) \varphi^{2}k^2|\omega|^{p}.
	\end{split}
\end{equation}

However, the second term on the right hand side, by Lemma \ref{1121}, we get(c.f. \cite{zbMATH06827101}) 
\begin{align}\label{dhfhjdhf}
	&\frac{p}{2p-2}\int_{M}\bigg\langle d ( |\omega|^{p-2}\omega), d  \omega\bigg\rangle  \nonumber \\
	&\leq \frac{2(p-2)}{p-1}\int_{M} \Big| \nabla |\omega|^{\frac{p}{2}}\Big|  \phi |\nabla \phi||\omega|^{\frac{p}{2}}.
\end{align}

Thus, we get 

\begin{equation}\label{ew}
	\begin{split}
	&\left( \frac{2-p}{p}+\frac{p}{2p-2} A_{p,n} \frac{4(p-1)^2}{p^2}-\epsilon\right)	\int_{M \backslash B_{R}(o)} \varphi^{2}	 |\nabla| \omega|^{\frac{p}{2}} |^{2}\\
		\leq& (1+\frac{(p-2)^2}{\epsilon (p-1)^2})\int_{M \backslash B_{R}(o)}|\nabla \varphi|^{2}|\omega|^{p}
					+\int_{M \backslash B_{R}(o)} \varphi^{2}|\omega|^{p}\\
					&+\left(\frac{p}{2p-2}\frac{\sqrt{n-1}}{2}-1\right)\int_{M \backslash B_{R}(o)}\varphi^{2}S|\omega|^{p}.
	\end{split}
\end{equation}
As  in \cite{dung2020harmonic}, we also have 
\begin{equation}\label{}
	\begin{split}
				\int_{M \backslash B_{R}(o)} \varphi^{2}|\omega|^{p} & \leq \frac{1}{\lambda_{1}(M)} \int_{M \backslash B_{R}(o)}\left|\nabla\left(\varphi|\omega|^{\frac{p}{2}}\right)\right|^{2} \\
			& \leq \frac{1}{\lambda_{1}(M)}\left(1+\frac{1}{\epsilon}\right) \int_{M \backslash B_{R}(o)}|\nabla \varphi|^{2}|\omega|^{ p}+\frac{1+\epsilon}{\lambda_{1}(M)} \int_{M \backslash B_{R}(o)} \varphi^{2}|\nabla| \omega|^{\frac{p}{2}}|^{2}.
			\end{split}
\end{equation}
So we get
\begin{equation}\label{wq}
	\begin{split}
	&	\left[ 1-B^{-1}\left( \frac{p}{2p-2}+\frac{1}{a+1}\right)k^2\frac{1+\epsilon}{\lambda_{1}(M)}\right]	\int_{M \backslash B_{R}(o)} \varphi^{2}	 | \omega|^{p} \\
		\leq & \bigg[\frac{1}{\lambda_{1}(M)}\left(1+\frac{1}{\epsilon}\right)+\frac{1+\epsilon}{\lambda_{1}(M)}B^{-1}(1+\frac{(p-2)^2}{\epsilon (p-1)^2})\bigg]\int_{M \backslash B_{R}(o)}|\nabla \varphi|^{2}|\omega|^{p},
	\end{split}
\end{equation}
where $ B=\frac{2-p}{p}+\frac{p}{2p-2} A_{p,n} \frac{4(p-1)^2}{p^2}-\epsilon. $ It suffice to prove that 
\begin{equation*}
	\begin{split}
		\frac{2-p}{p}+\frac{p}{2p-2} A_{p,n} \frac{4(p-1)^2}{p^2}>0.
	\end{split}
\end{equation*}
which  is 

\begin{equation}\label{g}
	\begin{split}
		\frac{2-p}{p}+\frac{p}{2p-2}  \frac{4(p-1)^2}{p^2}	\frac{1}{(p-1)^2}\min\{1, \frac{(p-1)^2}{n-1}\}>0
	\end{split}
\end{equation}

If $ 2\leq p \leq 1+\sqrt{n-1},$
\begin{equation}\label{e}
	\begin{split}
		\frac{2-p}{p}+\frac{p}{2p-2} \frac{1}{n-1} \frac{4(p-1)^2}{p^2}>0.
	\end{split}
\end{equation}
which holds if $ p < 2\frac{n-2}{n-3}$. Thus, we can choose $ \epsilon $ such that $ B>0. $

If  $ p\geq 1+\sqrt{n-1}, $ we have 
\begin{equation}\label{f}
	\begin{split}
		\frac{2-p}{p}+\frac{p}{2p-2}  \frac{4}{p^2}=\frac{2-p}{p}+\frac{2}{p(p-1)}  >0.
	\end{split}
\end{equation}
which  holds if $ p<3 $. Thus, we can choose $ \epsilon $ such that $ B>0. $

Take $r_{0}>r_{1},$ we can choose cutoff fuction $\phi$ with supp $\phi \subseteq M-B_{x_{0}}(r_{1})$,  and  sufficiently small  $\epsilon_3$, 
which implies that
\begin{align}\label{897}
	\int_{M}|\nabla |\omega|^{\frac{p}{2}}|^{2}\phi^{2}
	\leq \widetilde{D}\int_{M}|\nabla \phi|^{2} |\omega|^{p},
\end{align}
where $\widetilde{D}$ is a constant depending on $ p $. Using the Young inequality   and  the Hoffman-Spruck inequality(cf.\cite{6}) , as in \cite{dung2020harmonic}, we also have 
\begin{equation}\label{8753}
	\begin{split}
		\left(\int_{M \backslash B_{R}(o)}\left(\varphi|\omega|^{\frac{p}{2}}\right)^{\frac{2 n}{n-2}}\right)^{\frac{n-2}{n}} & \leq C_{s} \int_{M \backslash B_{R}(o)}\left|\nabla\left(\varphi|\omega|^{\frac{p}{2}}\right)\right|^{2}+C_{s} \int_{M \backslash B_{R}(o)} \varphi^{2}|\omega|^{ p} H^{2} \\
		& \leq(1+s) C_{s} \int_{M \backslash B_{R}(o)} \varphi^{2}|\nabla| \omega|^{\frac{p}{2}}|^{2}+\left(1+\frac{1}{s}\right) C_{s} \int_{M \backslash B_{R}(o)}|\nabla \varphi|^{2}|\omega|^{ p} \\
		&+C_{s} \int_{M \backslash B_{R}(o)} \varphi^{2}|\omega|^{p} H^{2}
	\end{split}
\end{equation}
where $ C_s $ is a constant in Sobolev inequality.

By \eqref{ew} and \eqref{8753}, 
\begin{equation}\label{238}
	\begin{split}
		&\left(\int_{M \backslash B_{R}(o)}\left(\varphi|\omega|^{\frac{p}{2}}\right)^{\frac{2 n}{n-2}}\right)^{\frac{n-2}{n}}\\
		\leq& A^{-1}(1+s)C_s\bigg[1+\frac{(p-2)^2}{4\epsilon (p-1)^2}\int_{M \backslash B_{R}(o)}|\nabla \varphi|^{2}|\omega|^{p}
		+\left( \frac{p}{2p-2}+\frac{1}{a+1}\right)k^2\int_{M \backslash B_{R}(o)} \varphi^{2}|\omega|^{p}\\
		&+\left(\frac{p}{2p-2}\frac{\sqrt{n-1}}{2}-1\right)\int_{M \backslash B_{R}(o)}\varphi^{2}S|\omega|^{p}\bigg]\\
		&+\left(1+\frac{1}{s}\right) C_{s} \int_{M \backslash B_{R}(o)}|\nabla \varphi|^{2}|\omega|^{ p} +C_{s} \int_{M \backslash B_{R}(o)} \varphi^{2}|\omega|^{p} H^{2}
	\end{split}
\end{equation}
Using $ |H |^2\leq n S $, and 
\begin{equation*}
	\begin{split}
		\frac{p}{2p-2}\frac{\sqrt{n-1}}{2}-1<0
	\end{split}
\end{equation*}
we can choose sufficiently large $ s $ such that 
\begin{equation*}
	\begin{split}
	A^{-1}(1+s)\left( 	\frac{p}{2p-2}\frac{\sqrt{n-1}}{2}-1\right) +n \leq 0.
	\end{split}
\end{equation*}
It follows from \eqref{238},
\begin{equation}\label{}
	\begin{split}
		&\left(\int_{M \backslash B_{R}(o)}\left(\varphi|\omega|^{\frac{p}{2}}\right)^{\frac{2 n}{n-2}}\right)^{\frac{n-2}{n}}\\
		\leq &A^{-1}(1+s)C_s\bigg[1+\frac{(p-2)^2}{4\epsilon (p-1)^2}\int_{M \backslash B_{R}(o)}|\nabla \varphi|^{2}|\omega|^{p}
		+\left( \frac{p}{2p-2}+\frac{1}{a+1}\right)k^2\int_{M \backslash B_{R}(o)} \varphi^{2}|\omega|^{p}
		\bigg]\\
		&+\left(1+\frac{1}{s}\right) C_{s} \int_{M \backslash B_{R}(o)}|\nabla \varphi|^{2}|\omega|^{ p} \\
		\end{split}
\end{equation}
By \eqref{wq}, we get
\begin{align*}
	&(\int_{M \backslash B_{R}(o)}(\phi|\omega|^{\frac{p}{2}})^{\frac{2m}{m-2}}dx)^{\frac{m-2}{m}}
	\leq C_1\int_{M \backslash B_{R}(o)}|\nabla\phi|^{2}|\omega|^{p}.
\end{align*}

As in \cite{dung2020harmonic}, we  choose cutoff function  $\phi  $ on   $M$ such that for $ r> R+1 $
\begin{equation}
	\begin{cases}
		&  \text{  $0\leq \phi\leq 1$ } \\
		&  \text{$\phi=1$, on $B_{r}(o)\backslash B_{R+1}(o)$    }\\
		&  \text{ $\phi=0$, on $  B_R(o)\cup \left( M- B_{2r}(x_{0})\right) $, }\\
	&\text{$ |\nabla\phi| < C_2$, on $B_{R+1}(o)\backslash B_{R}(o)$  },\\
		&    \text{$ |\nabla\phi| < \frac{C_2}{R}$, on $B_{2r}(o)\backslash B_{r}(o)$  },
	\end{cases}   
\end{equation}
According to the definiton of $\phi$, we have
\begin{align*}
	&\bigg(\int_{B_{r}(o)\backslash B_{R+1}(o)}(|\omega|^{\frac{p}{2}})^{\frac{2m}{m-2}}dx\bigg)^{\frac{m-2}{m}}\leq P_{1}\int_{B_{R+1}(o)\backslash B_{R}(o)}|\omega|^{p}+\frac{P_{1}}{r^{2}}\int_{B_{2r}(o)\backslash B_{r}(o)}|\omega|^{p}.
\end{align*}
Then let $r\rightarrow \infty$, we have
\begin{align*}
	&\bigg(\int_{M\backslash B_{R+1}(o)}(|\omega|^{\frac{p}{2}})^{\frac{2m}{m-2}}dx\bigg)^{\frac{m-2}{m}}
	\leq P_{1}\int_{B_{R+1}(o)\backslash B_{R}(o))}|\omega|^{p}.
\end{align*}

It follows from the H\"{o}lder inequality(cf. Formula (30) in \cite{1} or \cite[(4.16)(4.17)]{dung2020harmonic}) that
\begin{align}\label{890}
	\int_{B_{R+2}(o)}|\omega|^{p}\leq  P_{2}\int_{B_{R+1}(o)}|\omega|^{p},
\end{align}
where $P_{2}$ depends on $ Vol(B_{R+2}(o)),m,p.$

Recall we have proved the following inequality,
\begin{equation*}
	\begin{split}
		|\omega|\Delta |\omega|^{p-1}= \frac{4}{(m-1)p^2}|\nabla| \omega|^{\frac{p}{2}} |^{2}-\bigg\langle \delta d(|\omega|^{p-2}\omega),\omega\bigg\rangle
		+|\omega|^{p}\alpha.
	\end{split}
\end{equation*}
where $ \alpha= \overline{\operatorname{BiRic}^{a}}\left(\frac{X}{|X|}, N\right)-\Phi_{a}(H, S)-a(\overline{\operatorname{Ric}}(N, N)+S).$  Once we have \eqref{897}\eqref{890}, the remained step is the sama as that in \cite{2} noticing that the difference of $ \alpha $ from that in \cite[(32)]{2} doesn't affect the correctness of the statement. Finally, let $V$ be any finite-dimensional subspace of $H^{1}(L^{p}(M))$. 

 By  Lemma  2.2 in \cite{2}, there exists $ \omega\in V $ such that
\begin{equation*}
	\mathrm{dim}(V) \int_{B_{R+1}(o)}|\omega|^{p} \leq \vol( B_{R+1}(o) )  \min(C_{p}(_{q}^{m}),dim V )\sup_{B_{(R+1)}(o)} |\omega|^{p}.
\end{equation*}
Thus
\begin{equation*}
	dim(V) \leq C.
\end{equation*}
where $ C $ depends on $ \Vol( B_{R+1}(o)) ,m,p, H, \alpha $

\end{proof}

Next, we consider the case where $ H=0. $
\begin{thm}
	Let $\bar{M}$ be an $(n+1)$-dimensional Hadamard manifold with $\overline{{\BiRic}}^{a}$ satisfying $-k^{2} \leq \overline{\operatorname{BiRic}}^{a}$, where a is a given nonnegative real number, $k$ is a nonzero constant. Let $M$ be a complete noncompact minimal hypersurface with finite index that is immersed in $\bar{M}$. 
	If one of the following two conditions hold, 
	
(1) $2 \leq p<\min\{1+\sqrt{n-1},4\frac{n}{n-1}, 2\frac{n-2}{n-3}\}$, assume that
\begin{equation*}
	\begin{split}
		\lambda_{1} > \left( \frac{2-p}{p}+\frac{2(p-1)^2}{p(n-1)} -\bigg(\frac{p}{2p-2}\left(\frac{n-1}{n}\right)-1\bigg)\right)^{-1} \left( \frac{p}{2p-2}k^2+\frac{p}{2p-2}\left(\frac{n-1}{n}\right) \frac{k^2}{a+1}\right) 
	\end{split}
\end{equation*}

(2) If $1+\sqrt{n-1}\leq p <3  ,$ we assume that 
 \begin{equation*}
	\begin{split}
		\lambda_{1} > \left( \frac{2-p}{p}+ \frac{2}{p(p-1)}-\bigg(\frac{p}{2p-2}\left(\frac{n-1}{n}\right)-1\bigg)\right)^{-1} \left( \frac{p}{2p-2}k^2+\frac{p}{2p-2}\left(\frac{n-1}{n}\right) \frac{k^2}{a+1}\right) 
	\end{split}
\end{equation*}
	Then,
	\[
	\operatorname{dim} \mathcal{H}^{1}\left(L^{ p}(M)\right)<\infty,
	\]
	where $\mathcal{H}^{1}\left(L^{ p}(M)\right)$ denotes the space of $L^{p}$ harmonic 1-forms on $M$.
\end{thm}

\begin{proof}Recall  the inequality we have proved
\begin{equation}\label{}
	\begin{split}
		|\omega|^{p-1}\Delta |\omega|^{p-1}\geq&  A_{p,n}|\nabla| \omega|^{p-1} |^{2}-\bigg\langle \delta d(|\omega|^{p-2}\omega),|\omega|^{p-2}\omega\bigg\rangle\\
		&+|\omega|^{2p-4}\left(-k^2-\frac{n-1}{n}S\right)|\omega|^{2}.
	\end{split}
\end{equation}
As in \eqref{eq3}, we have 
\begin{equation}
	\begin{split}
		&|\omega|^{\frac{p}{2}} \Delta|\omega|^{\frac{p}{2}}\\
		\geq& 	\left( \frac{2-p}{p}+\frac{p}{2p-2} A_{p,n} \frac{4(p-1)^2}{p^2}\right) |\nabla| \omega|^{\frac{p}{2}} |^{2}-\frac{p}{2p-2}\bigg\langle \delta d(|\omega|^{p-2}\omega),\omega\bigg\rangle\\
		&+\frac{p}{2p-2}\left(-k^2-\frac{n-1}{n}S\right)|\omega|^{p}.
	\end{split}
\end{equation}
Multiplying both sides by $ \phi^2 $
\begin{equation}\label{eqq}
	\begin{split}
			& \int_M \phi^2	\left( \frac{2-p}{p}+\frac{p}{2p-2} A_{p,n} \frac{4(p-1)^2}{p^2}\right) |\nabla| \omega|^{\frac{p}{2}} |^{2}-\frac{p}{2p-2}\int_M \phi^2\bigg\langle \delta d(|\omega|^{p-2}\omega),\omega\bigg\rangle\\
		&+\int_M \phi^2\frac{p}{2p-2}\left(-k^2-\frac{n-1}{n}S\right)|\omega|^{p}\\
		\leq & \int_M \phi^2|\omega|^{\frac{p}{2}} \Delta|\omega|^{\frac{p}{2}}\\
		=&\int_M |\nabla \phi |^2 |\omega|^p-\int_M |\nabla(\phi |\omega|^p) |^2.
	\end{split}
\end{equation}

However, by $-k^{2} \leq \overline{\operatorname{BiRic}}^{a}$, we know
\begin{equation}\label{}
	\begin{split}
	&	\int_{M \backslash B_{R}(o)}(-\frac{k^2}{a+1}+S) \varphi^{2}|\omega|^{ p}\\	
	\leq& \int_{M \backslash B_{R}(o)}(\overline{\operatorname{Ric}}(N)+S) \varphi^{2}|\omega|^{ p}\\  \leq& \int_{M \backslash B_{R}(o)}\left|\nabla\left(\varphi|\omega|^{\frac{p}{2}}\right)\right|^{2} .
	\end{split}
\end{equation}

Thus, we get
\begin{equation}\label{eqq3}
	\begin{split}
		&	\int_{M \backslash B_{R}(o)}(S \varphi^{2}|\omega|^{ p}	
		 \leq \int_{M \backslash B_{R}(o)}\left|\nabla\left(\varphi|\omega|^{\frac{p}{2}}\right)\right|^{2} + \frac{k^2}{a+1}\int_{M \backslash B_{R}(o)}\varphi^{2}|\omega|^{ p}.
	\end{split}
\end{equation}

Combining \eqref{eqq} and \eqref{eqq3}, we have 
\begin{equation}
	\begin{split}
		& \int_M \phi^2	\left( \frac{2-p}{p}+\frac{p}{2p-2} A_{p,n} \frac{4(p-1)^2}{p^2}\right) |\nabla| \omega|^{\frac{p}{2}} |^{2}\\
		&-\frac{p}{2p-2}\int_M \phi^2\bigg\langle \delta d(|\omega|^{p-2}\omega),\omega\bigg\rangle-\int_M \frac{p}{2p-2}k^2\phi^2|\omega|^{p}\\
		\leq&\int_M |\nabla \phi |^2 |\omega|^p-\int_M |\nabla(\phi |\omega|^p) |^2\\
		&+\frac{p}{2p-2}\left(\frac{n-1}{n}\right)\bigg[\int_{M \backslash B_{R}(o)}\left|\nabla\left(\varphi|\omega|^{\frac{p}{2}}\right)\right|^{2} + \frac{k^2}{a+1}\int_{M \backslash B_{R}(o)}\varphi^{2}|\omega|^{ p}\bigg],
	\end{split}
\end{equation}

which can be written as

\begin{equation}
	\begin{split}
		& \int_M \phi^2	\left( \frac{2-p}{p}+\frac{p}{2p-2} A_{p,n} \frac{4(p-1)^2}{p^2}\right) |\nabla| \omega|^{\frac{p}{2}} |^{2}\\
		&-\bigg[\frac{p}{2p-2}k^2+\frac{p}{2p-2}\left(\frac{n-1}{n}\right) \frac{k^2}{a+1}\bigg]\int_M \phi^2|\omega|^{p}-\int_M |\nabla \phi |^2 |\omega|^p\\
		\leq&\bigg[\frac{p}{2p-2}\left(\frac{n-1}{n}\right)-1\bigg]\int_{M \backslash B_{R}(o)}\left|\nabla\left(\varphi|\omega|^{\frac{p}{2}}\right)\right|^{2} +\frac{2p}{p-1}\frac{p-2}{p}\int_{M} \Big| \nabla |\omega|^{\frac{p}{2}}\Big|  \phi |\nabla \phi||\omega|^{\frac{p}{2}}.
	\end{split}
\end{equation}

Thus, we have 
\begin{equation}\label{m4}
	\begin{split}
		& \left( \frac{2-p}{p}+\frac{p}{2p-2} A_{p,n} \frac{4(p-1)^2}{p^2}-\epsilon_{1}\right)\int_M \phi^2	 |\nabla| \omega|^{\frac{p}{2}} |^{2}\\
		&-\int_M |\nabla \phi |^2 |\omega|^p\\
		\leq&\bigg[\frac{p}{2p-2}\left(\frac{n-1}{n}\right)-1+\bigg[\frac{p}{2p-2}k^2+\frac{p}{2p-2}\left(\frac{n-1}{n}\right) \frac{k^2}{a+1}\bigg]\frac{1}{\lambda_{1}(M)}\bigg]\int_{M \backslash B_{R}(o)}\left|\nabla\left(\varphi|\omega|^{\frac{p}{2}}\right)\right|^{2} \\
		&+2\frac{p-2}{p-1} \frac{1}{4\epsilon_{1}}\int_M |\nabla \phi |^2 |\omega|^p.
	\end{split}
\end{equation}

Following  \cite{dung2020harmonic}, if $ \frac{p}{2p-2}\left(\frac{n-1}{n}\right)-1+\bigg[\frac{p}{2p-2}k^2+\frac{p}{2p-2}\left(\frac{n-1}{n}\right) \frac{k^2}{a+1}\bigg]\frac{1}{\lambda_{1}(M)} \leq 0, $ then  by \eqref{m4}, we have 
\begin{equation*}
	\begin{split}
			& \left( \frac{2-p}{p}+\frac{p}{2p-2} A_{p,n} \frac{4(p-1)^2}{p^2}-\epsilon_{1}\right)\int_M \phi^2	 |\nabla| \omega|^{\frac{p}{2}} |^{2}\\
			\leq & 2\frac{p-2}{p-1} \frac{1}{4\epsilon_{1}}\int_M |\nabla \phi |^2 |\omega|^p.
	\end{split}
\end{equation*}
By the conditions on $ p $ and \eqref{g}\eqref{e}\eqref{f},  we have 
\begin{equation*}
	\begin{split}
		\frac{2-p}{p}+\frac{p}{2p-2} A_{p,n}\frac{4(p-1)^2}{p^2}>0.
	\end{split}
\end{equation*}

If $ \frac{p}{2p-2}\left(\frac{n-1}{n}\right)-1+\bigg[\frac{p}{2p-2}k^2+\frac{p}{2p-2}\left(\frac{n-1}{n}\right) \frac{k^2}{a+1}\bigg]\frac{1}{\lambda_{1}(M)} >0, $  then  by \eqref{m4}, we have
\begin{equation}
	\begin{split}
		& \left( \frac{2-p}{p}+\frac{p}{2p-2} A_{p,n} \frac{4(p-1)^2}{p^2}-\epsilon_{1}\right)\int_M \phi^2	 |\nabla| \omega|^{\frac{p}{2}} |^{2}\\
		\leq
		&\left( 1+2\frac{p-2}{p-1} \frac{1}{4\epsilon_{1}}\right) \int_M |\nabla \phi |^2 |\omega|^p\\
		&+\bigg[\frac{p}{2p-2}\left(\frac{n-1}{n}\right)-1+\bigg[\frac{p}{2p-2}k^2+\frac{p}{2p-2}\left(\frac{n-1}{n}\right) \frac{k^2}{a+1}\bigg]\frac{1}{\lambda_{1}(M)}\bigg]\\
		&\bigg((1+\frac{1}{\epsilon_2})\int_M |\nabla \phi |^2 |\omega|^p+ (1+\epsilon_2)\int_M | \phi |^2 |\nabla|\omega|^{\frac{p}{2}}|^2\bigg).
	\end{split}
\end{equation}
It follows that
\begin{equation}\label{plo}
	\begin{split}
		& \left( \frac{2-p}{p}+\frac{p}{2p-2} A_{p,n} \frac{4(p-1)^2}{p^2}-\epsilon_{1}-B(1+\epsilon_2)\right)\int_M \phi^2	 |\nabla| \omega|^{\frac{p}{2}} |^{2}\\
		\leq&
		\left( 1+2\frac{p-2}{p-1} \frac{1}{4\epsilon_{1}}\right) \int_M |\nabla \phi |^2 |\omega|^p+B(1+\frac{1}{\epsilon_2})\int_M |\nabla \phi |^2 |\omega|^p,
	\end{split}
\end{equation}
where $ B=\bigg[\frac{p}{2p-2}\left(\frac{n-1}{n}\right)-1+\bigg[\frac{p}{2p-2}k^2+\frac{p}{2p-2}\left(\frac{n-1}{n}\right) \frac{k^2}{a+1}\bigg]\frac{1}{\lambda_{1}(M)}\bigg]. $

Note  it suffices to prove that we can choose $ \epsilon_{1},\epsilon_2 $ such that 
\begin{equation*}
	\begin{split}
		 &\frac{2-p}{p}+\frac{p}{2p-2} A_{p,n} \frac{4(p-1)^2}{p^2}-\epsilon_{1}-B(1+\epsilon_2)\\
		  =&\frac{2-p}{p}+\frac{1}{2p-2}  \frac{4}{p}-\bigg\{\frac{p}{2p-2}\left(\frac{n-1}{n}\right)-1\\
		  &+\bigg[\frac{p}{2p-2}k^2+\frac{p}{2p-2}\left(\frac{n-1}{n}\right) \frac{k^2}{a+1}\bigg]\frac{1}{\lambda_{1}(M)}\bigg\}-\epsilon_{1}-B\epsilon_2>0.
	\end{split}
\end{equation*}
If $ p > 1+\sqrt{n-1} ,$ 
\begin{equation*}
	\begin{split}
		\frac{2-p}{p}+\frac{1}{2p-2} \frac{4}{p}-\bigg(\frac{p}{2p-2}\left(\frac{n-1}{n}\right)-1\bigg)>0.
	\end{split}
\end{equation*}
which holds if $ p<\frac{4n}{n-1}. $  Moreover, 
\begin{equation*}
	\begin{split}
		\lambda_{1} >& \left( \frac{2-p}{p}+\frac{1}{2p-2}  \frac{4}{p}-\bigg(\frac{p}{2p-2}\left(\frac{n-1}{n}\right)-1\bigg)\right)^{-1}\\
		&\times \left( \frac{p}{2p-2}k^2+\frac{p}{2p-2}\left(\frac{n-1}{n}\right) \frac{k^2}{a+1}\right) .
	\end{split}
\end{equation*}
 Thus in the case where  $ p > 1+\sqrt{n-1} ,$ we can choose $ \epsilon_{1},\epsilon_2 $ such that 
\begin{equation*}
	\begin{split}
		&\frac{2-p}{p}+\frac{p}{2p-2} A_{p,n} \frac{4(p-1)^2}{p^2}-\epsilon_{1}-A(1+\epsilon_2)>0.
			\end{split}
\end{equation*}

If $ p \leq 1+\sqrt{n-1} ,$  then
\begin{equation*}
	\begin{split}
		&\frac{2-p}{p}+\frac{p}{2p-2} A_{p,n} \frac{4(p-1)^2}{p^2}-\epsilon_{1}-A(1+\epsilon_2)\\
		=&\frac{2-p}{p}+\frac{1}{2p-2} \frac{(p-1)^2}{(m-1)} \frac{4}{p}\\
		&-\bigg\{\frac{p}{2p-2}\left(\frac{n-1}{n}\right)-1+\bigg[\frac{p}{2p-2}k^2+\frac{p}{2p-2}\left(\frac{n-1}{n}\right) \frac{k^2}{a+1}\bigg]\frac{1}{\lambda_{1}(M)}\bigg\}-\epsilon_{1}-A\epsilon_2.
	\end{split}
\end{equation*}
In this case, if
\begin{equation*}
	\begin{split}
		\frac{2-p}{p}+\frac{1}{2p-2} \frac{(p-1)^2}{(n-1)} \frac{4}{p}-\bigg(\frac{p}{2p-2}\left(\frac{n-1}{n}\right)-1\bigg)>0.
	\end{split}
\end{equation*}
which holds if $ p<\frac{4n}{n-1} $. In additon, 
\begin{equation*}
	\begin{split}
		\lambda_{1} >& \left( \frac{2-p}{p}+ \frac{2(p-1)}{p(n-1)} -\bigg(\frac{p}{2p-2}\left(\frac{n-1}{n}\right)-1\bigg)\right)^{-1}\\
		&\times \left( \frac{p}{2p-2}k^2+\frac{p}{2p-2}\left(\frac{n-1}{n}\right) \frac{k^2}{a+1}\right) ,
	\end{split}
\end{equation*}
  Thus in the case where  $ p \leq  1+\sqrt{n-1} ,$ we can choose $ \epsilon_{1},\epsilon_2 $ such that 
\begin{equation*}
	\begin{split}
		&\frac{2-p}{p}+\frac{p}{2p-2} A_{p,n} \frac{4(p-1)^2}{p^2}-\epsilon_{1}-A(1+\epsilon_2)>0.
	\end{split}
\end{equation*}

Hence, from \eqref{plo}, we get 
\begin{align}
	\int_{M}|\nabla |\omega|^{\frac{p}{2}}|^{2}\phi^{2}
	\leq C\int_{M}|\nabla \phi|^{2} |\omega|^{p}.
\end{align}

The remained proof  is the same as that in Theorem \ref{p1.2}.
\end{proof}

\end{document}